\documentclass[12pt]{amsart}
\usepackage{amsmath,amsthm,amssymb,graphicx,color}
\usepackage{epsfig,amsfonts,latexsym, hyperref, bbm}
\begin{document}

\allowdisplaybreaks
\newcommand{\eqdef}{\stackrel{d}{=}}
\newcommand{\eqlaw}{\stackrel{\footnotesize{\mathcal{L}}}{=}}
\newcommand{\iid}{\stackrel{\footnotesize{iid}}{\sim}}
\newcommand{\eqas}{\stackrel{a.s.}{=}}
\newcommand{\probconv}{\stackrel{P}{\longrightarrow}}
\newcommand{\asconv}{\stackrel{a.s.}{\longrightarrow}}
\newcommand{\weakconv}{\stackrel{\footnotesize{\mathcal{L}}}{\longrightarrow}}
\newcommand{\bX}{\mathbf{X}}
\newcommand{\R}{\mathbb{R}}
\newcommand{\Q}{\mathbb{Q}}
\newcommand{\Zd}{\mathbb{Z}^d}
\newcommand{\SaS}{S \alpha S}
\newcommand{\mC}{\mathcal{C}}
\newcommand{\mD}{\mathcal{D}}

\newtheorem{thm}{Theorem}[section]
\newtheorem{remark}[thm]{Remark}
\newtheorem{propn}[thm]{Proposition}
\newtheorem{qn}[thm]{Question}
\newtheorem{exc}{Exercise}
\newtheorem{problem}[thm]{Problem}
\newtheorem{prop}[thm]{Property}
\newtheorem{example}[thm]{Example}
\newtheorem{defn}[thm]{Definition}
\newtheorem{claim}[thm]{Claim}
\newtheorem{lemma}[thm]{Lemma}
\newtheorem{cor}[thm]{Corollary}
\newtheorem{ann}[thm]{Announcement}

\numberwithin{equation}{section}

\newcommand{\cd}{\ \stackrel{d}{\longrightarrow} \ }
\newcommand{\ed}{\ \stackrel{d}{=} \ }
\newcommand{\cp}{\ \stackrel{\bP}{\rightarrow} \ }
\newcommand{\cas}{\ \mathop{\longrightarrow}\limits^{\mbox{a.s.}}\ }
\newcommand{\eas}{\ \stackrel{\mbox{a.s.}}{=} \ }
\newcommand{\cl}{\ \mathop{\longrightarrow}\limits^{{\mathcal L}_1}\ }
\newcommand{\cll}{\ \mathop{\longrightarrow}\limits^{{\mathcal L}_2}\ }
\newcommand{\clas}{\ \mathop{\longrightarrow}\limits_{\mbox{a.s.}}
                                                     ^{{\mathcal L}_1}\ }
\newcommand{\cllas}{\ \mathop{\longrightarrow}\limits_{\mbox{a.s.}}
                                                     ^{{\mathcal L}_2}\ }

\newcommand{\XZd}{\{X_t\}_{t \in \mathbb{Z}^d}}
\newcommand{\XZ}{\{X_t\}_{t \in \mathbb{Z}}}
\newcommand{\XRd}{\{X_t\}_{t \in \mathbb{R}^d}}
\newcommand{\XR}{\{X_t\}_{t \in \mathbb{R}}}

\newcommand{\TT}{\mbox{${\mathcal T}$}}
\newcommand{\PP}{\mbox{${\mathcal P}$}}
\newcommand{\PPP}{\mbox{${\mathfrak P}$}}
\newcommand{\CC}{\mbox{${\mathcal C}$}}
\newcommand{\EE}{\mbox{${\mathcal E}$}}
\newcommand{\VV}{\mbox{${\mathcal V}$}}
\newcommand{\UU}{\mbox{${\mathcal U}$}}
\newcommand{\GG}{\mbox{${\mathcal G}$}}
\newcommand{\FF}{\mbox{${\mathcal F}$}}
\newcommand{\SSS}{\mbox{${\mathcal S}$}}
\newcommand{\NN}{\mbox{${\mathcal N}$}}
\newcommand{\MM}{\mbox{${\mathcal M}$}}
\newcommand{\HH}{\mbox{${\mathcal H}$}}
\newcommand{\AAA}{\mbox{${\mathcal A}$}}
\newcommand{\BB}{\mbox{${\mathcal B}$}}
\newcommand{\II}{\mbox{${\mathcal I}$}}
\newcommand{\LL}{\mbox{${\mathcal L}$}}

\newcommand{\eps}{\varepsilon}
\newcommand{\bone}{{\mathbf 1}}

\newcommand{\bA}{{\bf A}}
\newcommand{\bE}{{\bf E}}
\newcommand{\bP}{{\bf P}}
\newcommand{\bJ}{{\bf J}}
\newcommand{\bQ}{{\bf Q}}
\newcommand{\bI}{{\bf I}}
\newcommand{\bY}{{\bf Y}}
\newcommand{\bC}{{\bf C}}
\newcommand{\bT}{{\bf T}}
\newcommand{\ba}{{\bf a}}
\newcommand{\bb}{{\bf b}}
\newcommand{\bx}{{\bf x}}
\newcommand{\bt}{{\bf t}}
\newcommand{\bi}{{\bf i}}
\newcommand{\bj}{{\bf j}}
\newcommand{\by}{{\bf y}}
\newcommand{\bp}{{\bf p}}
\newcommand{\btheta}{{\mathbf{\theta}}}

\newcommand{\bzero}{\bf 0}

\newcommand{\phit}{\{\phi_t\}_{t \in \mathbb{Z}^d}}

\newcommand{\TTT}{\mbox{${\mathscr{T}}$}}
\newcommand{\EEE}{\mbox{${\mathscr{E}}$}}
\newcommand{\VVV}{\mbox{${\mathscr{V}}$}}

\newcommand{\MMM}{\mbox{${\mathfrak{M}}$}}

\newcommand{\VVVV}{\mbox{$\widetilde{\VVV}$}}

\newcommand{\Pbb}{\mbox{${\mathbb P}$}}

\newcommand{\var}{{\bf Var}}
\newcommand{\conc}{{\rm conc}\,}
\newcommand{\sfrac}[2]{{\textstyle\frac{#1}{#2}}}
\newcommand{\Ex}{\text{\bf E\/}}
\newcommand{\pr}{\mbox{{\bf Pr}}}
\newcommand{\med}{\mbox{median}}
\newcommand{\mudiag}{\mu^\nearrow}
\newcommand{\gen}{{\rm gen}}
\newcommand{\HT}{{\rm ht}}
\newcommand{\pcrit}{p_{\mbox{{\tiny crit}}}}
\newcommand{\Hbar}{{\overline H}}
\newcommand{\stleq}{\preccurlyeq}
\newcommand{\point}{\mathit{\Delta}}
\newcommand{\lleq}{\leq^*}
\newcommand{\Mopt}{\MM_{{\rm opt}}}
\newcommand{\muprod}{\mu \otimes \mu}
\newcommand{\Tprod}{T \otimes T}
\newcommand{\norm}{\parallel}

\newcommand{\Rbold}{{\mathbb{R}}}
\newcommand{\Tbold}{{\mathbb{T}}}
\newcommand{\Zbold}{{\mathbb{Z}}}
\newcommand{\Nbold}{{\mathbb{N}}}
\newcommand{\Ebold}{{\mathbb{E}}}
\newcommand{\Qbold}{{\mathbb{Q}}}

\newcommand{\bMM}{\stackrel{\rightarrow}{\MM}}
\newcommand{\ef}{\stackrel{\rightarrow}{e}}
\newcommand{\eb}{\stackrel{\leftarrow}{e}}
\newcommand{\dE}{\stackrel{\rightarrow}{\bE}}
\newcommand{\dEE}{\stackrel{\rightarrow}{\EE}}
\newcommand{\bMopt}{\stackrel{\rightarrow}{\Mopt}}

\title{Maxima of stable random fields, nonsingular actions and finitely generated abelian groups: A survey}
\dedicatory{Dedicated to Professor B.~V.~Rao}
\author{Parthanil Roy}
       \address{Theoretical Statistics and Mathematics Unit, Indian Statistical Institute, 8th Mile, Mysore Road, RVCE Post, Bengaluru 560059, India.}
       \email{parthanil.roy@gmail.com}
       \thanks{Parthanil Roy was supported by Cumulative Professional Development Allowance from Ministry of Human Resource Development, Government of India and the project RARE-318984 (a Marie Curie FP7 IRSES Fellowship).}
        \maketitle

\begin{abstract}
This is a self-contained introduction to the applications of ergodic theory of nonsingular (also known as quasi-invariant) group actions and the structure theorem for finitely generated abelian groups on the extreme values of stationary symmetric stable random fields indexed by $\mathbb{Z}^d$. It is based on a mini course given in the \href{http://www.isid.ac.in/~lps16/PastWebsites/LPS8/index.html}{\emph{Eighth Lectures on Probability and Stochastic Processes}} (held in the \emph{Bangalore Centre of Indian Statistical Institute} during December 6-10, 2013) except that a few recent references have been added in the concluding part. This article is a survey of existing work and the proofs are therefore skipped or briefly outlined.
\end{abstract}

\section{Introduction}

The main goal of this paper is to study \emph{symmetric $\alpha$-stable} (S$\alpha$S) random fields with a view to formalizing the phrase \emph{long range dependence} (also known as \emph{long memory}), a property observed in many real life processes. This property typically refers to dependence between observations $X_t$ far separated in $t$. Historically, it was first
observed by a famous British hydrologist Harold Edwin Hurst, who
noticed an empirical phenomenon (now known as Hurst phenomenon; see
\cite{hurst:1951} and \cite{hurst:1955}) while looking at
measurements of the water flow in the Nile River.

A series of papers of Benoit Mandelbrot and his co-workers tried to
explain Hurst phenomenon using long range dependence; see \cite{mandelbrot:wallis:1968} and \cite{mandelbrot:wallis:1969c}.
From then on, processes having long memory have been used in many
different areas including economics, internet
modelling, climate studies, linguistics, DNA sequencing, etc. For a detailed discussion
on long range dependence, see \cite{samorodnitsky:2006} and the
references therein.

Most of the classical definitions
of long range dependence appearing in the literature are based on the
second order properties (e.g., covariance, spectral density,
variance of partial sum, etc.) of stochastic processes. For
example, one of the most widely accepted definitions of this notion for a stationary Gaussian process is that \emph{a stationary Gaussian process has long range dependence if its correlation function decays slowly enough to make it not
summable}.  In the heavy tails context, however, this definition
becomes ambiguous because correlation function may not even exist and even if it exists, it may not have enough
information about the dependence structure of the process.

In the context of stationary $S\alpha S$ processes ($0 < \alpha <
2$) indexed by $\mathbb{Z}$, instead of looking for a substitute for correlation function,
\cite{samorodnitsky:2004a} suggested a new approach through phase
transition phenomena as follows. Suppose that $(P_\theta,\,\theta
\in \Theta)$ is a family of laws of a stationary stochastic process,
where $\theta$ is a parameter of the process lying in a parameter
space $\Theta$. If $\Theta$ can be partitioned into $\Theta_0$ and
$\Theta_1$ in such a way that a significant number of functionals of this stochastic process change
dramatically as we pass from $\Theta_0$ to $\Theta_1$, then this
phase transition can be thought of as a change from short memory to
long memory. The aforementioned paper investigates the rate of growth of the partial maxima of the stationary $S\alpha S$ process indexed by $\mathbb{Z}$. A transition boundary is observed based on the
ergodic theoretic properties of the underlying nonsingular $\mathbb{Z}$-action obtained from the seminal work \cite{rosinski:1995}. In this article, we shall discuss the main results of these papers and their extensions (see \cite{rosinski:2000} and \cite{roy:samorodnitsky:2008}) to the S$\alpha$S random fields.

This survey paper is organised as follows. In Sections~\ref{sec:sas_distn} and~\ref{sec:sas_rm_int}, we follow \cite{samorodnitsky:taqqu:1994} and briefly discuss S$\alpha$S random variables and vectors, S$\alpha$S random measures, and the integrals with respect to them. Integral representations of S$\alpha$S random fields indexed by $\mathbb{Z}^d$ are studied in Section~\ref{sec:int_repn_sas_rf} and the stationary case is investigated in Section~\ref{sec_rosinski_repn}. We present a decomposition of stationary S$\alpha$S random fields into two independent components based on the Hopf decomposition of the underlying nonsingular $\mathbb{Z}^d$-actions in Section~\ref{sec:cons_diss_parts}. This decomposition is then connected, in Section~\ref{sec:maxima}, to the asymptotic behaviour of a partial maxima sequence of these fields. Section~\ref{sec:fg_ab_gp} deals with applications of the structure theorem for finitely generated abelian groups in this context and a brief discussion of open problems. Finally in Section~\ref{sec:other_works}, we carry out an extensive literature survey of related work.

\section{Symmetric $\alpha$-stable Distributions and Random Fields} \label{sec:sas_distn}

This section (and the next one) contains standard materials on stable distributions and random measures as given in \cite{samorodnitsky:taqqu:1994}. The only difference is that we specialise the results in the symmetric case.

\begin{defn} A random variable X is said to follow \underline{symmetric $\alpha$-stable} (S$\alpha$S) distribution ($\alpha \in (0, 2]$ is called the index of stability) with scale parameter $\sigma>0$ (denoted by $X \sim S\alpha S(\sigma)$) if its characteristic function is of the form
\[
E(e^{i\theta X})=e^{-\sigma^\alpha |\theta|^\alpha}, \; \theta \in \mathbb{R}.
\]
\end{defn}

It is not difficult to check that this is indeed a valid characteristic function; see, for example, \cite{feller:1971}.

\begin{prop}\label{prop_known_distns}(a) If $\alpha=1$, then $X \sim$ Cauchy distribution with density function $f_X(x)=\frac{\sigma}{\pi(x^2+\sigma^2)}$, $-\infty<x<\infty$.\\
(b) If $\alpha=2$, then $X \sim N(0, 2\sigma^2)$.
\end{prop}

These are the only two cases in which the density functions are known in closed form. For the other values of $\alpha$, $X$ is supported on $\mathbb{R}$ with a continuous density function that can be written in a series. See, for example, \cite{ibragimov:linnik:1971}, \cite{feller:1971} and \cite{zolotarev:1986}. 

It is worth mentioning that $X$ behaves very differently when $\alpha =2$ in comparison to the case $0 < \alpha < 1$. For example, in the latter situation, $X$ has infinite second moment (see Corollary~\ref{cor_moment} below) while in the former case it is Gaussian and hence has all moments finite. We shall assume from now on that $0 <\alpha < 2$.

\begin{prop} If $X_i \sim S\alpha S(\sigma_i)$, $i=1,2$ and $X_1 , X_2$ are independent, then $a_1 X_1+a_2 X_2\sim S\alpha S \big((|a_1|^\alpha \sigma_1^\alpha+|a_2|^\alpha \sigma_2^\alpha)^{1/\alpha}\big)$. In particular, $X \eqlaw -X$.
\end{prop}

\begin{prop} If $X_1, X_2, \ldots, X_n \iid S\alpha S(\sigma)$, then $\sum_{i=1}^n X_i \eqlaw n^{1/\alpha} X_1$.
\end{prop}

\begin{prop} \label{prop_tail_behaviour} If $X \sim S\alpha S(\sigma)$ with $\alpha \in (0,2)$, then $P(|X|>\lambda) \sim \sigma^\alpha C_\alpha \lambda^{-\alpha}$ as $\lambda \to \infty$, where \begin{equation}
C_\alpha = {\left(\int_0^\infty x^{-\alpha} \sin{x}\,dx
\right)}^{-1}
 =\left\{
  \begin{array}{ll}
  \frac{1-\alpha}{\Gamma(2-\alpha) \cos{(\pi
\alpha/2)}}&\mbox{\textit{\small{if }}}\alpha \neq 1,\\
  \frac{2}{\pi}
&\mbox{\textit{\small{if }}}\alpha = 1.
  \end{array}
 \right. \label{defn_C_alpha}
\end{equation}
\end{prop}

\begin{proof}[Sketch of Proof] For $\alpha=1$, this is trivial to prove. For $\alpha \in (0,1)$, we divide the proof into several steps as described below.

\textbf{Step 1.} The Laplace transform of $|X|$ is $E(e^{-\gamma |X|})=\exp{\left(-\frac{\sigma^\alpha}{\cos{(\pi \alpha/2)}}\gamma^\alpha\right)}$, $\gamma \geq 0$. (Use Proposition 1.2.12 and Property 1.2.13 of \cite{samorodnitsky:taqqu:1994}.) \smallskip

\textbf{Step 2.} Using integration by parts,
\[
\int_0^\infty e^{-\gamma \lambda} P(|X|>\lambda) d\lambda = \frac{1-E(e^{-\gamma |X|})}{\gamma} \sim \frac{\sigma^\alpha}{\cos{(\pi \alpha/2)}} \gamma^{\alpha-1}
\]
as $\gamma \to 0$.  \smallskip

\textbf{Step 3.} Step 2 and Theorem XIII.5.4 of \cite{feller:1971} imply that $P(|X|>\lambda) \sim \frac{\sigma^\alpha}{\cos{(\pi \alpha/2)}\Gamma(1-\alpha)} \lambda^{-\alpha}=\sigma^\alpha C_\alpha \lambda^{-\alpha}$ since $0<\alpha<1$.

See \cite{feller:1971} and \cite{samorodnitsky:taqqu:1994} for the details in the $0<\alpha<1$ case and the proof in the $1<\alpha<2$ case.
\end{proof}

\begin{cor} \label{cor_moment} For $0 < \alpha < 2$, $E|X|^p <\infty$ if $0 < p < \alpha$ and $E|X|^p =\infty$ if $p \geq \alpha$.
\end{cor}

The following series representation of an S$\alpha$S random variable will be extremely useful for us later in this survey.

\begin{thm} \label{thm_series_repn_SalphaS} Let $\{\epsilon_i\}_{i \geq 1}$, $\{\Gamma_i\}_{i \geq 1}$, $\{W_i\}_{i \geq 1}$ be three independent sequences of random variables, where $\epsilon_1, \epsilon_2, \ldots \i.i.d. \pm 1$ with probability $1/2$ each, $\Gamma_1 < \Gamma_2 < \cdots $ are the arrival times of a homogeneous Poisson process with unit arrival rate, and $W_1, W_2, \ldots$ are i.i.d. satisfying $E|W_1|^\alpha<\infty$. Then the series
\begin{equation}
\sum_{i=1}^\infty \epsilon_i \Gamma_i^{-1/\alpha} W_i \label{series_repn_SalphaS}
\end{equation}
converges almost surely to a random variable $X \sim S\alpha S\left((C_\alpha^{-1}E|W_1|^\alpha)^{1/\alpha}\right)$.
\end{thm}

\begin{remark} \label{remark_series_repn_SalphaS} \textnormal{It can be shown that $P\big(|\epsilon_1\Gamma_1^{-1/\alpha} W_1|>\lambda\big) \sim E|W_1|^\alpha \lambda^{-\alpha}$ as $\lambda \to \infty$ whereas $P\big(\sum_{i=2}^\infty\epsilon_i \Gamma_i^{-1/\alpha} W_i>\lambda\big)=o(\lambda^{-\alpha})$ as $\lambda \to \infty$; see pg 26-28 of \cite{samorodnitsky:taqqu:1994}. According to the discussions in pg 26 of this reference, the first term $\epsilon_1\Gamma_1^{-1/\alpha} W_1$ is the dominating term (of \eqref{series_repn_SalphaS}) that gives the precise asymptotics of its tail while the rest of the terms provide the ``necessary corrections for the whole sum to have an $\alpha$-stable distribution''. This is regarded as the \emph{one large jump heuristic} for an S$\alpha$S random variable.}
\end{remark}

\begin{proof}[Sketch of Proof of Theorem~\ref{thm_series_repn_SalphaS}] \textbf{Step 1.} Three series theorem (\cite{feller:1971}, Theorem IX.9.3) can be used to show that the series \eqref{series_repn_SalphaS} converges almost surely as $n \to \infty$. This is not completely straightforward but somewhat routine; see pg 24-25 of \cite{samorodnitsky:taqqu:1994}. \smallskip

\textbf{Step 2.} Use the following ``cool trick'' from elementary probability theory to identify the distribution of the (almost surely) convergent series \eqref{series_repn_SalphaS}. Take a sequence of $U_1, U_2, \ldots \iid Unif(0,1)$ independent of $\{\epsilon_i\}_{i \geq 1}$ and $\{W_i\}_{i \geq 1}$. Recall that for each $n$,
\[
\left(\frac{\Gamma_1}{\Gamma_{n+1}}, \frac{\Gamma_2}{\Gamma_{n+1}}, \ldots, \frac{\Gamma_n}{\Gamma_{n+1}} \right) \eqlaw (U_{(1)}, U_{(2)}, \ldots, U_{(n)}),
\]
where $U_{(1)} < U_{(2)} <\cdots < U_{(n)}$ are the order-statistics obtained from the random sample $(U_1, U_2, \ldots, U_n)$. Using this equality of distribution and an exchangeability argument,
\begin{equation}
\left(\frac{\Gamma_{n+1}}{n}\right)^{1/\alpha}\sum_{i=1}^n \epsilon_i \Gamma_i^{-1/\alpha} W_i \eqlaw \frac{1}{n^{1/\alpha}}\sum_{i=1}^n \epsilon_i U_{(i)}^{-1/\alpha} W_i \eqlaw \frac{1}{n^{1/\alpha}}\sum_{i=1}^n \epsilon_i U_{i}^{-1/\alpha} W_i \label{equality_of_distn}.
\end{equation}
It is not difficult to verify that $\{\epsilon_i U_i^{-1/\alpha} W_i\}_{i \geq 1}$ is a sequence of i.i.d. symmetric random variables satisfying $P(|\epsilon_1 U_1^{-1/\alpha} W_1|>\lambda) \sim E|W_1|^\alpha \lambda^{-\alpha}$ as $\lambda \to \infty$. Therefore by domain of attraction condition for stable distributions (see Section~XVII.5 of \cite{feller:1971}), strong law of large numbers and \eqref{equality_of_distn}, Theorem~\ref{thm_series_repn_SalphaS} follows.
\end{proof}

\begin{defn} A random vector $\bX:=(X_1, X_2, \ldots, X_k)$ is said to follow multivariate S$\alpha$S distribution if each nondegenerate linear combination $\sum_{i=1}^k c_i X_i$ $(c_1, c_2, \ldots, c_k \in \mathbb{R})$ follows S$\alpha$S distribution. In this case, $\bX$ is called an S$\alpha$S random vector.
\end{defn}

The following result gives a very nice and useful characterization of an S$\alpha$S random vector.

\begin{thm}\label{thm_ch_fn_SalphaS_vector} $\bX \in \mathbb{R}^k$ is an S$\alpha$S random vector with $0<\alpha<2$ if and only if there exists a unique finite symmetric measure $\Gamma$ on the unit sphere $S_k:=\{\bx: \|\bx\|_2=1\}$ such that
\begin{equation}
E(e^{i \btheta^{T}\bx})=\exp\left\{-\int_{S_k}|\btheta^{T}\bx|^\alpha\,\Gamma(d\bx)\right\}. \label{form_ch_fn_SalphaS_vector}
\end{equation}
\end{thm}

\begin{proof} See \cite{kuelbs:1973}.
\end{proof}

\begin{defn} The measure $\Gamma$ as in \eqref{form_ch_fn_SalphaS_vector} is called the spectral measure of the S$\alpha$S random vector $\bX$.
\end{defn}

\section{S$\alpha$S Random Measures and Integrals} \label{sec:sas_rm_int}

We shall now introduce S$\alpha$S random measures and integral with respect to such measures. In fact, we shall first introduce the integral and then define the random measure. Let $(E, \mathcal{E}, m)$ be a $\sigma$-finite measure space, $0<\alpha<2$ and
\[
F:=L^\alpha(E, \mathcal{E}, m)=\left\{f:E \to \mathbb{R}: \|f\|_\alpha<\infty\right\},
\]
where 
\[
\|f\|_\alpha:=\left(\int_E |f|^\alpha dm\right)^{1/\alpha}.
\]
Note that $F$ is a Banach space when $1 \leq \alpha <2$ (but not a Hilbert space) with the norm $\|\cdot\|_\alpha$. However for $0<\alpha<1$, $\|\cdot\|_\alpha$ is not even a norm and hence $F$ has very little structure. It is a metric space with the distance function $d_\alpha(f,g):=\|f-g\|_\alpha^\alpha$. In particular, $F$ is a very rigid space for all $\alpha \in (0,2)$ in the sense that it has very few isometries. We shall exploit this rigidity in the second half of this article.

Roughly speaking, our next goal is to define an S$\alpha$S process $\{I(f): f \in F\}$ indexed by $F$ so that $M(A):=I(\mathbbm{1}_A)$, $A \in \mathcal{E}_0:=\{A \in \mathcal{E}: m(A)<\infty\}$ becomes an ``S$\alpha$S random measure'' and $I(f)$ becomes the ``integral with respect to $M$''. We attain this goal as follows. Given $f_1, f_2, \ldots, f_k \in F$, we define a probability measure $P_{f_1, f_2,\ldots, f_k}$ on $\mathbb{R}^k$ by its characteristic function
\begin{equation}
\psi_{f_1, f_2,\ldots, f_k}=\exp{\left\{-\Big\|\sum_{j=1}^k \theta_j f_j\Big\|_\alpha^\alpha\right\}}. \label{form_ch_fn_fd_If}
\end{equation}

\begin{propn} \label{propn_psi_f_well_defined} For any $f_1, f_2, \ldots, f_k \in F$, $\psi_{f_1, f_2,\ldots, f_k}$ is the characteristic function of an S$\alpha$S random vector. In particular, $P_{f_1, f_2,\ldots, f_k}$ is well-defined.
\end{propn}

\begin{proof}  Let $E_{+}:=\left\{x \in E: \sum_{j=1}^k \big(f_j(x)\big)^2>0\right\}$. Define a measure $\Gamma$ on the unit sphere $S_k$ as
\[
\Gamma(A):= \frac{1}{2} \int_{\pi(A)}\Big(\sum_{j=1}^k f_j^2\Big)^{\alpha/2} dm + \frac{1}{2} \int_{\pi(-A)} \Big(\sum_{j=1}^k f_j^2\Big)^{\alpha/2} dm,\;\;A \subseteq S_k,
\]
where
\[
\pi(A):= \left\{x \in E_{+}: \left(\frac{f_1(x)}{\sqrt{\sum_{j=1}^k \big(f_j(x)\big)^2}},  \ldots, \frac{f_k(x)}{\sqrt{\sum_{j=1}^k \big(f_j(x)\big)^2}}\right) \in A\right\}.
\]
It is easy to check that $\Gamma$ is a symmetric finite measure on $S_k$ such that $\psi_{f_1, f_2,\ldots, f_k}$ is of the form \eqref{form_ch_fn_SalphaS_vector}. This completes the proof.
\end{proof}

From Proposition~\ref{propn_psi_f_well_defined} and Kolmogorov extention theorem, it follows that there exists an S$\alpha$S process $\{I(f): f \in F\}$ with finite-dimensional distributions of the form \eqref{form_ch_fn_fd_If}. In particular, each $I(f) \sim S\alpha S(\|f\|_\alpha)$.

\begin{propn}[$I$ is linear and independently scattered] \label{exc_If_lin_ind_sct} For all functions $f_1, f_2, \ldots, f_k \in F$ and for all $a_1, a_2, \ldots a_k\in \mathbb{R}$,
\[
I(a_1 f_1 + a_2 f_2 + \cdots + a_k f_k) = a_1 I(f_1) + a_2 I(f_2) + \cdots + a_k I(f_k)
\]
almost surely. If further $f_1, f_2, \ldots, f_k$ have pairwise disjoint support, then $I(f_1), I(f_2), \ldots, I(f_k)$ are independent.
\end{propn}

\begin{defn}
Let $(E, \mathcal{E}, m)$ be a $\sigma$-finite measure space. A collection $\{M(A): A \in \mathcal{E}_0\}$ of random variables defined on the same probability space is called an S$\alpha$S random measure on $E$ with control measure $m$ if
\begin{enumerate}
\item each $M(A) \sim S\alpha S\big((m(A))^{1/\alpha}\big)$,
\item if $A_1, A_2, \ldots, A_k \in \mathcal{E}_0$ are pairwise disjoint, then $M(A_1), M(A_2),$ $\ldots$, $M(A_k)$ are independent (i.e., $M$ is independently scattered),
\item if $A_1, A_2, \ldots$ are pairwise disjoint such that $\bigcup_{i=1}^\infty A_i \in \mathcal{E}_0$, then $M(\bigcup_{i=1}^\infty A_i)= \sum_{i=1}^\infty M(A_i)$ almost surely (i.e., $M$ is $\sigma$-additive).
\end{enumerate}

\begin{propn} For every $\sigma$-finite measure space $(E, \mathcal{E}, m)$ there exists an S$\alpha$S random measure on $E$ with control measure $m$.
\end{propn}

\begin{proof} Define $M(A):=I(\mathbbm{1}_A)$, $A \in \mathcal{E}_0$. All the properties of an $S\alpha S$ random measure follows from the properties of $I$ mentioned above except the $\sigma$-additivity, which can be established as follows. Note that finite additivity follows from Proposition~\ref{exc_If_lin_ind_sct} and therefore,
\[
M\Big(\bigcup_{i=1}^\infty A_i\Big) - \sum_{i=1}^n M(A_i) \eqas M\Big(\bigcup_{i=n+1}^\infty A_i\Big) \sim S\alpha S \left(\Big(\sum_{i=n+1}^\infty m(A_i)\Big)^{1/\alpha}\right).
\]
The above observation yields, by Levy's continuity theorem, that the partial sum $\sum_{i=1}^n M(A_i) \probconv M(\bigcup_{i=1}^\infty A_i)$, which in turn implies that this convergence holds almost surely since $M(A_1), M(A_2), \ldots$ are independent.
\end{proof}

\end{defn}

\normalsize

Here is a result that gives the motivation behind thinking $I(f)$ as an ``integral of $f$ with respect to $M$''.

\begin{thm} $\{I(f): f \in F\}$ defined above satisfies the following properties.
\begin{enumerate}
\item If $f \in F$ is a simple function of the form $f=\sum_{j=1}^k c_j \mathbbm{1}_{A_j}$ with pairwise disjoint $A_1, A_2, \ldots, A_k \in \mathcal{E}_0$, then by linearity of $I$,
    \[
    I(f) \eqas \sum_{j=1}^k c_j I(\mathbbm{1}_{A_j}) =  \sum_{j=1}^k c_j M(A_j)\,.
    \]
\item Let $f \in F$ be any function (not necessarily simple). Take a sequence of simple functions $\{f_n\}_{n\geq 1}$ such that $f_n \asconv f$ and $|f_n| \leq g$ for some $g \in F$ (such a sequence always exists for any $f \in F$), then $I(f_n) \probconv I(f)$.
\end{enumerate}
\end{thm}

\begin{proof} The first part follows trivially from linearity of $I$. For the second part (including existence of such a sequence), see pg 122 - 124 of \cite{samorodnitsky:taqqu:1994}.
\end{proof}

In view of the above result, we shall denote $I(f)$ by $\int_E f dM$ for $f \in F$. This motivates the following definition.
\begin{defn}
The S$\alpha$S process $\{I(f)\}_{f \in F}$ is called the integral with respect to the random measure $M$ and this is denoted by
\begin{equation}
\{I(f)\}_{f\in F} \eqlaw \left\{\int_E f(x) M(dx)\right\}_{f\in F}. \label{defn_int_wrt_M}
\end{equation}
\end{defn}
\noindent We would like to emphasize that the notation \eqref{defn_int_wrt_M} is a fancy way of writing that $\{I(f)\}_{f \in F}$ is a stochastic process indexed by $F$ such that for any $f_1, f_2, \ldots, f_k \in F$, the joint characteristic function of $\big(I(f_1), I(f_2), \ldots, I(f_k)\big)$ is given by \eqref{form_ch_fn_fd_If}. In other words, the integral w.r.t. the S$\alpha$S random measure is defined only in distribution. 

\begin{remark} \label{remark_I_on_F_0} \textnormal{For any $F_0 \subseteq F$, we can use the notation $\left\{\int_E f dM\right\}_{f \in F_0}$ to denote the S$\alpha$S process $\{I(f)\}_{f \in F_0}$. This remark will be useful later in this paper because we shall always work with a ``suitably chosen'' proper subset of $F$.}
\end{remark}

\begin{example} [S$\alpha$S Levy Motion] \textnormal{Take $E = [0,\infty)$, $m=Leb$ (the Lebesgue measure on $[0,\infty)$) and let $M$ be an S$\alpha$S random measure on $[0,\infty)$ with control measure $Leb$. This $M$ is called S$\alpha$S Levy Random Measure.}

\textnormal{Observe that $X_t:=M\big([0,t]\big)$, $t \geq 0$ is an S$\alpha$S process satisfying the following properties:}
\begin{enumerate}
\item \textnormal{$X_0 \eqas 0$.}
\item \textnormal{$\{X_t\}_{t \geq 0}$ has independent increments, i.e., for all $0 \leq t_1 < t_2 < \cdots \,t_k <\infty$, $X_{t_1}, X_{t_2}-X_{t_1}, X_{t_3}-X_{t_2}, \ldots, X_{t_k}-X_{t_{k-1}}$ are independent. (Follows from the fact that $M$ is independently scattered.)}
\item \textnormal{For all $0 \leq s < t < \infty$, $X_t - X_s \sim S\alpha S\big((t-s)^{1/\alpha}\big)$. In particular, $\{X_t\}_{t \geq 0}$ has stationary increments, i.e., for all $\tau \geq 0$,}
    \[
    \{X_t - X_0\}_{t \geq 0} \eqlaw \{X_{t+\tau}-X_\tau\}_{t \geq 0}.
    \]
\item \textnormal{$\{X_t\}_{t \geq 0}$ is self-similar with index ${1}/{\alpha}$, i.e., for all $c>0$,}
    \[
    \{X_{ct}\}_{t \geq 0} \eqlaw \{c^{1/\alpha} X_{t}\}_{t \geq 0}
    \]
\end{enumerate}
\textnormal{$\{X_t\}_{t \geq 0}$ described above is called an S$\alpha$S Levy motion. It is the analogue of Brownian motion in the S$\alpha$S world.}
\end{example}

\section{Integral Representation of an S$\alpha$S Random Field} \label{sec:int_repn_sas_rf}

\begin{defn} A stochastic process $\{X_t\}_{t \in T}$ is called an S$\alpha$S process (indexed by $T$) if all of its finite-dimensional distributions are multivariate S$\alpha$S distributions. When $T=\mathbb{Z}^d$ or $\mathbb{R}^d$ for some $d \in \mathbb{N}$, $\{X_t\}_{t \in T}$ is called an S$\alpha$S random field.
\end{defn}

From now on, we shall only deal with S$\alpha$S random fields. In order to keep life simple, we shall concentrate on the discrete parameter case, i.e., $T=\mathbb{Z}^d$ for some $d \geq 1$ .

\begin{defn}  Let $(S, \mathcal{S}, \mu)$ be a $\sigma$-finite standard Borel space. A family of functions $\{f_t\}_{t \in \mathbb{Z}^d} \subseteq L^\alpha(S, \mathcal{S}, \mu)$ is called an integral representation of an S$\alpha$S random field $\XZd$ if
\begin{equation}
\XZd \eqlaw \left\{\int_S f_t(s)M(ds)\right\}_{t \in \Zd}, \label{defn_int_repn}
\end{equation}
where $M$ is an S$\alpha$S random measure on $S$ with control measure $\mu$.
\end{defn}

Recall that \eqref{defn_int_repn} simply means that for all $t_1, t_2, \ldots, t_k \in \Zd$,
\begin{equation}
E\big(e^{i\sum_{j=1}^k \theta_j X_{t_j}}\big) = \exp{\left\{-\Big\|\sum_{j=1}^k \theta_j f_{t_j}\Big\|_\alpha^\alpha\right\}}, \; \theta_1, \theta_2, \ldots, \theta_k \in \R. \label{meaning_of_int_repn}
\end{equation}

\normalsize

\begin{thm} \label{thm_existence_int_repn} Every S$\alpha$S random field $\XZd$ has an integral representation.
\end{thm}

\begin{proof} See \cite{bretagnolle:dacunha-castelle:krivine:1966}, \cite{schreiber:1972}, \cite{schilder:1970}. See also \cite{kuelbs:1973} and \cite{hardin:1982b} for a discussion of history of \eqref{defn_int_repn}.
\end{proof}

For any integral representation $\{f_t\}_{t \in \mathbb{Z}^d} \subseteq L^\alpha(S, \mathcal{S}, \mu)$ , one can assume without loss of generality that
\[
\bigcup_{t \in \Zd} \mbox{Support}(f_t) \eqas S.
\]
From now on, we shall assume that this full support condition holds for all of our integral representations.

The converse of Theorem~\ref{thm_existence_int_repn} holds, i.e., given any $\sigma$-finite measure space $(S, \mathcal{S}, \mu)$, a family of functions $\{f_t\}_{t \in \mathbb{Z}^d} \subseteq L^\alpha(S, \mathcal{S}, \mu)$ and an S$\alpha$S random measure $M$ on $S$ with control measure $\mu$, we can construct an S$\alpha$S random field $\XZd$ using \eqref{defn_int_repn}. This follows trivially from Remark~\ref{remark_I_on_F_0} with $F_0=\{f_t: t \in \mathbb{Z}^d\}\subseteq L^\alpha(S, \mathcal{S}, \mu)$. Using this, one can construct many S$\alpha$S random fields, one of which is discussed below.

\begin{example}[Stationary S$\alpha$S Moving Average Random Field]\label{example_mixed_mov_avg}  \textnormal{This example was introduced (in the $d=1$ case) in \cite{surgailis:rosinski:mandrekar:cambanis:1993}. Let $(W, \mathcal{W}, \nu)$ be a $\sigma$-finite measure space. Define $S=W \times \Zd$ and $\mu=\nu \otimes \eta$, where $\eta$ is the counting measure on $\Zd$. Let $M$ be an S$\alpha$S random measure on $W \times \Zd$ with control measure $\nu \otimes \eta$. Take a single function $f \in L^\alpha(W \times \Zd, \nu \otimes \eta)$ and define a family $\{f_t\}_{t \in \mathbb{Z}^d}$ of functions as}
\[
f_t(w, s)=f(w, s+t), \; \;(w,s) \in W \times \Zd.
\]
\textnormal{It is easy to check that each $f_t \in L^\alpha(W \times \Zd, \nu \otimes \eta)$. The S$\alpha$S random field}
\begin{eqnarray}
\XZd&:\eqlaw& \left\{\int_{W \times \Zd} f_t(w,s) \,dM(w,s)\right\}_{t \in \Zd} \nonumber\\
          &\eqlaw& \left\{\int_{W \times \Zd} f(w,s+t) \,dM(w,s)\right\}_{t \in \Zd} \label{defn_mixed_mov_avg}
\end{eqnarray}
\textnormal{is called a stationary S$\alpha$S moving average random field.}
\end{example}
\begin{defn} \label{defn_stationary}A random field $\XZd$ is called stationary if $\XZd$ $\eqlaw \{X_{t+\tau}\}_{t \in \Zd}$ for all $\tau \in \Zd$.
\end{defn}

It is easy to verify that $\XZd$ defined by \eqref{defn_mixed_mov_avg} is stationary. If $W$ is a singleton, then $\XZd$ is a moving average random field with i.i.d. S$\alpha$S innovations. In view of this observation, one can think of $\XZd$ defined by \eqref{defn_mixed_mov_avg} as a mixture of moving averages and hence it is called a mixed moving average, which will play a very important role in this survey.

The following notion (introduced in \cite{hardin:1982b}) is extremely technical and yet useful. We shall first give the definition and then state a theorem that will help us understand its meaning.

\begin{defn} An integral representation $\{f_t\}_{t \in \mathbb{Z}^d} \subseteq L^\alpha(S, \mathcal{S}, \mu)$ of an S$\alpha$S random field is called a minimal representation if for all $B \in \mathcal{S}$, there exists $A \in \sigma\big\{f_t/f_{t^\prime}: t, t^\prime \in \Zd\big\}$ such that $\mu(A \Delta B)=0$.
\end{defn}

The ratio $f_t(s)/f_{t^\prime}(s)$ is defined to be $\infty$ when $f_t(s) \geq 0$, $f_{t^\prime}(s)=0$ and $-\infty$ when $f_t(s)<0$, $f_{t^\prime}(s)=0$. In particular, the $\sigma$-algebra $\sigma\big\{f_t/f_{t^\prime}: t, t^\prime \in \Zd\big\}$ is generated by a bunch of extended real-valued functions.

\begin{thm} \label{thm_existence_minimal_repn} Every S$\alpha$S random field has a minimal representation.
\end{thm}

The following result provides better insight into the notion of minimality of integral representations.

\begin{thm} \label{thm_rigidity_repn} Let $\{f^\ast_t\}_{t \in \mathbb{Z}^d} \subseteq L^\alpha(S^\ast, \mathcal{S}^\ast, \mu^\ast)$ be a minimal representation of an S$\alpha$S random field $\XZd$ and $\{f_t\}_{t \in \mathbb{Z}^d} \subseteq L^\alpha(S, \mathcal{S}, \mu)$ be any integral representation of $\XZd$. Then there exist measurable functions $\Phi: S \to S^\ast$ and $h: S \to \mathbb{R}\setminus \{0\}$ such that
\begin{equation}
\mu^\ast(A)=\int_{\Phi^{-1}(A)}|h|^\alpha d\mu, \;\,  A \in  \mathcal{S}^\ast, \label{reln_between_mu_h_muast}
\end{equation}
and for each $t \in \Zd$,
\begin{equation}
f_t(s)=h(s) f_t^\ast(\Phi(s)) \;\mbox{ for $\mu$-almost all } s \in S. \label{reln_between_f_and_fast}
\end{equation}
If further $\{f_t\}_{t \in \mathbb{Z}^d}$ is also a minimal representation, then $\Phi$ and $h$ are unique modulo $\mu$, $\Phi$ is one-to-one and onto, $\mu^\ast \circ \Phi \sim \mu$ and
\begin{equation}
|h|^\alpha=\frac{d (\mu^\ast \circ \Phi)}{d\mu}\;\;\mbox{ $\mu$-almost surely.}\label{reln_between_mu_h_muast_f_minimal}
\end{equation}
\end{thm}

\begin{proof}[Proofs of Theorems~\ref{thm_existence_minimal_repn} and \ref{thm_rigidity_repn}] These proofs use deep analysis of $L^\alpha$ spaces; see \cite{hardin:1981, hardin:1982b}. Theorem~\ref{thm_rigidity_repn}, for instance, follows from the rigidity (dearth of isometry) of $L^\alpha$ spaces, $0<\alpha<2$.
\end{proof}

Theorem~\ref{thm_rigidity_repn} provides some sort of uniqueness to integral representations of S$\alpha$S random fields and we shall capitalize on it heavily in this article. Since any integral representation can be expressed in terms of a minimal representation using \eqref{reln_between_f_and_fast}, $\{f^\ast_t\}_{t \in \Zd}$ should be regarded as a minimal element in the set of all integral representations. However it should be noted that in general, it is extremely difficult to check that a given integral representation is minimal. See \cite{rosinski:1994}, \cite{rosinski:1995} and \cite{rosinski:2006} for various useful results on minimal representations.

\normalsize

\section{The Stationary Case} \label{sec_rosinski_repn}

From now on, we shall assume that our S$\alpha$S random field $\XZd$ is stationary (see Definition~\ref{defn_stationary} above). Note that this means that for all $t_1, t_2, \ldots, t_k, \tau \in \Zd$ and for all $c_1, c_2, \ldots, c_k \in \mathbb{R}$, either $\sum_{i=1}^k c_i X_{t_i+\tau} \eqas 0$ or $\sum_{i=1}^k c_i X_{t_i+\tau}$ follows an $S\alpha S$ distribution whose scale parameter does not depend on $\tau$. The mixed moving average random field defined by \eqref{defn_mixed_mov_avg} serves as an important class of examples of such fields.

The ultimate goal of this survey is to study the asymptotic behaviour of a maxima sequence of $\XZd$ as $t$ varies in hypercubes of increasing size. More precisely, define for all $n \geq 1$,
$$
B_n =\big\{t=(t_1,t_2,\ldots,t_d) \in \Zd: \text{ each }t_i \in \{0, 1, 2, \ldots, n-1\}\big\},
$$
and
\begin{equation}
M_n:= \max_{t \in B_n} |X_t|, \; n \geq 1.
\end{equation}
We would like to answer, as much as possible, the following questions.
\begin{qn} What is the rate of growth of $M_n$ (as $n \to \infty$)?
\end{qn}
\begin{qn}  If we know the rate of growth of $M_n$, can we find its scaling limit?
\end{qn}

If $\{X_t\}_{t \in \Zd }\iid S\alpha S(\sigma)$, then by Proposition~1.11 of \cite{resnick:1987} and Property~\ref{prop_tail_behaviour} above, it follows that $M_n$ grows like $n^{d/\alpha}$ as $n \to \infty$ and $M_n/n^{d/\alpha} \weakconv a Z_\alpha$, where $a>0$ is a deterministic constant and $Z_\alpha$ is a Fr\'{e}chet type extreme value random variable with distribution function
\begin{equation}
P(Z_\alpha \leq z) =
\left\{
\begin{array}{ll}
e^{-z^{-\alpha}},& z>0,\\
0\,,                         & z \leq 0.
\end{array}
\right. \label{form_of_cdf_of_Zalpha}
\end{equation}
As long as the random field $\XZd$ has short memory, it is expected to exhibit the same rate of growth of $M_n$. On the other hand, if $\XZd$ has long memory, then $M_n$ is expected to grow slowly because this strong dependence will prevent erratic changes in the value of $X_t$ even when $\|t\|_\infty:=\max_{1 \leq i \leq n} |t_i|$ becomes large. We shall indeed observe a phase transition in the rate of growth of $M_n$ as $n \to \infty$. Because of the intuitions given above, this phase transition can be regarded as a passage from shorter memory to longer memory; see \cite{samorodnitsky:2004a} and \cite{roy:samorodnitsky:2008}.

In order to study the rate of growth of $M_n$, we need to know more about the integral representation of stationary S$\alpha$S random fields. It so happens that in the stationary case, any minimal representation of $\XZd$ has a very nice form in terms of a \emph{nonsingular $\Zd$-action} and an \emph{associated cocycle}. We introduce these terminologies below. See \cite{varadarajan:1970}, \cite{zimmer:1984}, \cite{krengel:1985} and \cite{aaronson:1997} for detailed discussions of these ergodic theoretic notions.

\begin{defn} Let $(S, \mathcal{S}, \mu)$ be a $\sigma$-finite standard Borel space. Then a family of measurable maps $\{\phi_t:S \to S\}_{t \in \Zd}$ is called a \underline{nonsingular} (also known as \underline{quasi-invariant}) $\Zd$-action if
\begin{enumerate}
\item $\phi_0(s)=s$ for $\mu$-almost all $s \in S$,
\item $\phi_{t_1+t_2} = \phi_{t_1} \circ \phi_{t_2}\;\mu$-almost surely,
\item $\mu \circ \phi_t^{-1} \sim \mu$  for all $t \in \Zd$.
\end{enumerate}
\end{defn}
\noindent In particular, if $\mu \circ \phi_t^{-1} = \mu$, then $\{\phi_t:S \to S\}_{t \in \Zd}$ is called a \emph{measure-preserving} $\Zd$-action. Clearly measure-preserving actions are nonsingular but the converse is not true. See \cite{aaronson:1997} for an example of a nonsingular $\mathbb{Z}$-action that is not measure-preserving.

\begin{example}\label{example_action_dissipative_Krengel} \textnormal{Let $(W, \mathcal{W}, \nu)$ be a $\sigma$-finite measure space, $S:=W \times \Zd$, $\mu:=\nu \otimes \eta$, where $\eta$ is the counting measure on $\Zd$. Define a $\mathbb{Z}^d$-action $\{\psi_t\}_{t \in \Zd}$ on $W \times \Zd$ as follows. For all $t \in \Zd$,}
\begin{equation}
\psi_t(w,s)=(w, s+t),\;\; (w,s) \in W \times \Zd.\label{defn_action_dissipative_Krengel}
\end{equation}
\textnormal{Clearly $\{\psi_t\}_{t \in \Zd}$ is a measure-preserving (and hence nonsingular) $\Zd$-action on $W \times \Zd$. Note that using this action, we can rewrite \eqref{defn_mixed_mov_avg} as
\begin{equation}
\XZd \eqlaw \left\{\int_{W \times \Zd} f(\psi_t(w,s)) \,dM(w,s)\right\}_{t \in \Zd}, \label{rewrite_mixed_mov_avg_psi_t}
\end{equation}
where $M$ is an S$\alpha$S random measure on $W \times \Zd$ with control measure $\nu \otimes \eta$.}
\end{example}

We need another notion that arises from cohomology theory and is widely used in ergodic theory.

\begin{defn} A collection of measurable maps $\big\{c_t: S \to \{-1, +1\}\big\}_{t \in \Zd}$ is called a ($\pm 1$-valued) cocycle for a nonsingular $\Zd$-action $\phit$ on $(S, \mathcal{S}, \mu)$ if for all $t_1, t_2 \in\Zd$,
\begin{equation}
c_{t_1 + t_2}(s) = c_{t_2}(s) c_{t_1}(\phi_{t_2}(s)) \label{defn_cocyle}
\end{equation}
for $\mu$-almost all $s \in S$.
\end{defn}

It was shown by Rosi\'{n}ski (see \cite{rosinski:1994}, \cite{rosinski:1995} and \cite{rosinski:2000}) that any minimal representation of a stationary S$\alpha$S random field can be written in terms of a nonsingular $\Zd$-action and an associated cocycle (see also the work \cite{hardin:1982b} that had expressed such a representation, in the $d=1$ case, using a group of linear isometries of $L^\alpha(S, \mathcal{S}, \mu)$ to itself). The seminal result of Rosi\'{n}ski is given below and should be considered as the key theorem of this paper.

\begin{thm} \label{thm_main_rosinski} Let $\{f_t\}_{t \in \Zd} \subseteq L^\alpha(S, \mathcal{S}, \mu)$ be a minimal representation of a stationary S$\alpha$S random field $\XZd$. Then there exist unique (modulo $\mu$) nonsingular $\Zd$-action $\phit$ on $(S, \mathcal{S}, \mu)$ and a $\pm 1$-valued cocycle $\{c_t\}_{t \in \Zd}$ for $\phit$ such that for all $t \in \Zd$,
\begin{equation}
f_t(s)=c_t(s) \big(f_0 \circ \phi_t(s)\big) \left(\frac{d (\mu \circ \phi_t)}{d \mu}(s)\right)^{1/\alpha} \;\mbox{ $\mu$-almost surely}. \label{repn_rosinski}
\end{equation}
\end{thm}

\begin{proof} The proof of this theorem is outlined in Section~\ref{subsec:proof_main_rosinski} below.
\end{proof}

The next theorem is the converse of Theorem~\ref{thm_main_rosinski} and can be used to produce many examples of stationary S$\alpha$S random fields.

\begin{thm} \label{thm_converse_main_rosinski} Take any measurable space $(S, \mathcal{S}, \mu)$, any $f \in L^\alpha(S, \mathcal{S}, \mu)$ any nonsingular $\Zd$-action $\phit$ on $(S, \mathcal{S}, \mu)$, and any $\pm 1$-valued cocycle $\{c_t\}_{t \in \Zd}$ for $\phit$. Then $\{f_t\}_{t \in \Zd} $ defined by \eqref{repn_rosinski} satisfies $\{f_t\}_{t \in \Zd} \subseteq L^\alpha(S, \mathcal{S}, \mu)$ and $\XZd$ defined by \eqref{defn_int_repn} (here $M$ is an S$\alpha$S random measure on $S$ with control measure $\mu$) is a stationary S$\alpha$S random field.
\end{thm}

\begin{proof} This result follows trivially from \eqref{meaning_of_int_repn}.
\end{proof}

\begin{defn} We introduce the phrase \underline{Rosi\'{n}ski representation} to mean any integral representation (not necessarily minimal) of the form \eqref{repn_rosinski}. In this case, we say that the stationary S$\alpha$S random field $\XZd$ is generated by the triplet $\left(f_0, \phit, \{c_t\}_{t \in \Zd}\right)$ on $(S, \mathcal{S},\mu)$.
\end{defn}

Note that the stationary mixed moving average S$\alpha$S random field defined by \eqref{defn_mixed_mov_avg} is generated by the triplet $\left(f, \{\psi_t\}_{t \in \Zd}, \{c_t \equiv 1\}_{t \in \Zd}\right)$ on $(W \times \Zd, \nu \otimes \eta)$ (here the notations are as in Example~\ref{example_action_dissipative_Krengel}). This means that \eqref{rewrite_mixed_mov_avg_psi_t} is a Rosi\'{n}ski representation with unit cocycle $c_t \equiv 1$ and unit Radon-Nikodym derivative $d((\nu \otimes \eta)\circ \psi_t)/d(\nu \otimes \eta) \equiv 1$ for all $t \in \Zd$. The unit Radon-Nikodym derivative is obtained because $\{\psi_t\}_{t \in \Zd}$  is a measure-preserving $\Zd$-action.

Any minimal representation is a Rosi\'{n}ski representation but not the converse. Also given a particular minimal representation, the underlying nonsingular $\Zd$-action and the associated cocycle are unique almost surely. However since minimal representation is not unique, Rosi\'{n}ski representation is not unique either. Because of Theorem~\ref{thm_rigidity_repn}, the underlying $\Zd$-actions (of different Rosi\'{n}ski representations) preserve many important ergodic theoretic properties. We shall introduce one such property in this article and discuss its implications for the length of memory (and rate of growth of the maxima sequence $M_n$) of a stationary S$\alpha$S random field.

\subsection{Sketch of proof of Theorem~\ref{thm_main_rosinski}} \label{subsec:proof_main_rosinski}

The idea of this proof is as follows. Stationarity means the the law of $\XZd$ is invariant under the shift action of $\Zd$ on $\mathbb{R}^{\Zd}$. This measure-preserving $\Zd$-action, when viewed at the integral representation level, naturally induces a nonsingular action on $S$ and an associated cocycle yielding \eqref{repn_rosinski}. The main steps of this proof are sketched below.

Fix $t \in \Zd$. Note that because of stationarity of $\{X_{\tau}\}_{\tau \in \Zd}$ and minimality of $\{f_{\tau}\}_{\tau \in \Zd}$, it follows that $\{f_{\tau+t}\}_{\tau \in \Zd}$ is also a minimal representation of $\{X_{\tau}\}_{\tau \in \Zd}$. Therefore by Theorem~\ref{thm_rigidity_repn}, there exist unique (modulo $\mu$) maps $\phi_t:S \to S$ (one-to-one and onto) and $h_t:S \to \mathbb{R}\setminus\{0\}$ such that for all $\tau \in \Zd$,
\begin{align}
&f_{\tau+t}=h_t\, f_\tau \circ \phi_t \mbox{\;\;\;$\mu$-almost surely, and}\label{reln_f_taut_and_f_tau} \\
&0<|h_t|=\left(\frac{d(\mu \circ \phi_t)}{d\mu}\right)^{1/\alpha}\mbox{\;$\mu$-almost surely, and}\label{form_of_h_t}
\end{align}

Define $c_t:=h_t/|h_t|$, $t \in \Zd$. Putting $\tau =0$ in \eqref{reln_f_taut_and_f_tau} and using \eqref{form_of_h_t}, we get that $\mu$-almost surely
\[
f_t = c_t\,f_0 \circ \phi_t\,\left(\frac{d(\mu \circ \phi_t)}{d\mu}\right)^{1/\alpha}, \;t \in \Zd.
\]
Fix $t_1, t_2 \in \Zd$. Evaluating $f_{\tau + t_1 + t_2}$ in two different ways and using Theorem~\ref{thm_rigidity_repn} (more precisely, the uniqueness of the maps), we can conclude that $\phit$ is a nonsingular $\Zd$-action on $(S, \mathcal{S}, \mu)$, $\{c_t\}_{t \in \Zd}$ is a $\pm 1$-valued cocycle for $\phit$, and they are both unique modulo $\mu$.

\section{Conservative and Dissipative Parts} \label{sec:cons_diss_parts}

When $\XZd$ is generated by $\left(f_0, \phit, \{c_t\}_{t \in \Zd}\right)$, this triplet can be thought of as a highly infinite-dimensional parameter that determines the dependence structure of $\XZd$ and hence has information about its length of memory. It so happens that $f_0$ and $\{c_t\}_{t \in \Zd}$ do not have too much information about the memory (this is somewhat expected because $f_0$ is just one function and $c_t$'s are just $\pm 1$-valued functions). The nonsingular $\Zd$-action $\phit$, on the other hand, has a lot of information on the length of memory. The next few definitions and results are motivated by this.

\begin{defn} Suppose $\phit$ is a nonsingular $\Zd$-action on  $(S,\mathcal{S}, \mu)$. A set $W_\ast \in  \mathcal{S}$ is called a wandering set (for $\phit$) if $\{\phi_t(W_\ast): t \in \Zd\}$ is a pairwise disjoint collection of subsets of $S$.
\end{defn}

Roughly speaking, wandering sets never come back to itself under the action. In Example~\ref{example_action_dissipative_Krengel}, take any $W_0 \subseteq W$ and any $t_0 \in \Zd$. Then $W_\ast:=W_0 \times \{t_0\}$ is a wandering set.

\normalsize

The following result (see Proposition~1.6.1 in \cite{aaronson:1997}) gives a decomposition of $S$ into two disjoint and invariant parts.

\begin{thm}[Hopf Decomposition] \label{thm_hopf_decomp} Suppose $\phit$ is a nonsingular $\Zd$-action on $(S,\mathcal{S},\mu)$. Then there exist unique (modulo $\mu$) subsets $\mathcal{C}, \mathcal{D} \in \mathcal{S}$ such that
\begin{enumerate}
\item $\mathcal{C} \cap \mathcal{D} = \emptyset$ modulo $\mu$,
\item $\mathcal{C} \cup \mathcal{D} = S$ modulo $\mu$,
\item $\mathcal{C}$ and $\mathcal{D}$ are invariant under the action $\phit$, i.e., for all $t \in \Zd$, $\phi_t(\mathcal{C}) = \mathcal{C}$ and $\phi_t(\mathcal{D}) =\mathcal{D}$ modulo $\mu$,
\item $\mathcal{C}$ has no wandering subset of positive measure, and
\item $\mathcal{D}=\bigcup_{t \in \Zd} \phi_t(W_\ast)$ modulo $\mu$ for some wandering set $W_\ast$.
\end{enumerate}
\end{thm}

\begin{defn} $\mathcal{C}$ and $\mathcal{D}$ are called the conservative and dissipative parts (of $\phit$), respectively. $\phit$ is called conservative if $S=\mathcal{C}$ modulo $\mu$ and dissipative if $S=\mathcal{D}$ modulo $\mu$.
\end{defn}

Roughly speaking, conservative actions keep coming back to its starting point whereas the dissipative actions keep moving away. An example of dissipative action is given by Example~\ref{example_action_dissipative_Krengel} with $W_\ast=W \times \{\mathbf{0}\}$ being a wandering set whose translates cover $S$ (see Theorem~\ref{thm_hopf_decomp} above). On the other hand, the following remark provides many examples of conservative actions.

\begin{remark} \label{exc_example_cons_action} \textnormal{Note that any measure-preserving $\Zd$-action on a finite measure space is necessarily conservative. In particular, if $\mu$ is a probability measure on $S=\R^{\Zd}$ such that under $\mu$, the coordinate field $\{\pi_t\}_{t \in \Zd}$ (defined by $\pi_t(\mathbf{x})=\mathbf{x}(t)$, $\mathbf{x} \in \mathbb{R}^{\Zd}$) is stationary, then the shift action $\{\zeta_t\}_{t \in \Zd}$ of $\Zd$ on $\R^{\Zd}$, defined by}
\begin{equation}
(\zeta_t \mathbf{x})(s)=\mathbf{x}(s+t),\;\; \mathbf{x} \in \R^{\Zd}, s \in \Zd,
\label{form_of_shift_action_on_RZd}
\end{equation}
\textnormal{is conservative.}
\end{remark}

The following result confirms that even though Rosi\'{n}ski representation is not unique, the rigidity result Theorem~\ref{thm_rigidity_repn} is kind towards the dissipativity and conservativity of the underlying nonsingular $\Zd$-actions.

\begin{propn} \label{propn_field_gen_by_cons_diss} If a stationary S$\alpha$S random field is generated by a conservative (dissipative, resp.) $\Zd$-action in one Rosi\'{n}ski representation, then in any other Rosi\'{n}ski representation of the field, the underlying action must be conservative (dissipative, resp.).
\end{propn}

\begin{proof} See \cite{rosinski:1995} (for $d=1$) and \cite{roy:samorodnitsky:2008} (for $d>1$).
\end{proof}

\begin{remark} \label{remark_cons_diss_memory} \textnormal{The stationary S$\alpha$S random fields generated by conservative $\Zd$-actions tend to have longer memory compared to the ones generated by dissipative (or more generally non-conservative) actions because conservative actions keep coming back and hence introduce stronger dependence among the $X_t$'s. This heuristic reasoning can be validated by the growth of $M_n$ as $n \to\infty$.}
\end{remark}

The following result gives structure to a stationary S$\alpha$S random field generated by a dissipative $\Zd$-action.

\begin{thm} \label{thm_diss_mixed_mov_avg} A stationary S$\alpha$S random field is generated by a dissipative $\Zd$-action if and only if it is a mixed moving average defined by \eqref{defn_mixed_mov_avg}.
\end{thm}

\begin{proof} [Main Idea of the Proof] The \emph{if part} follows from Proposition~\ref{propn_field_gen_by_cons_diss} and the fact that the $\Zd$-action \eqref{defn_action_dissipative_Krengel} is dissipative. The \emph{only if part} uses a very deep result (known as Krengel's Structure Theorem; see \cite{krengel:1969} for $d=1$, and \cite{rosinski:2000}, \cite{roy:samorodnitsky:2008} for $d>1$) that states that any dissipative nonsingular $\Zd$ action is ``isomorphic'' (in an appropriate sense) to the $\Zd$-action \eqref{defn_action_dissipative_Krengel}. Exploiting this isomorphism, one can change the underlying action to \eqref{defn_action_dissipative_Krengel}. However to replace the cocycle by the unit cocycle, one has to work harder. This part of the proof is slightly technical. See pg 1176 - 1177 of \cite{rosinski:1995} for the detailed proof.
\end{proof}

The Hopf decomposition of the underlying nonsingular actions induces a decomposition of the stationary S$\alpha$S random field into two independent stationary components as follows. Let $\{f_t\}_{t \in \Zd} \subseteq L^\alpha(S, \mu)$ be a Rosi\'{n}ski representation of a stationary S$\alpha$S random field $\XZd$ with underlying nonsingular $\Zd$-action $\phit$. Let $S = \mathcal{C}\cup \mathcal{D}$ be the Hopf decomposition for $\phit$. Then
\begin{equation}
X_t = \int_S f_t dM =  \int_\mathcal{C} f_t dM +  \int_\mathcal{D} f_t dM =:X^\mathcal{C}_t + X^\mathcal{D}_t, \;t \in \Zd, \label{decomp_X_t_rosinski}
\end{equation}
where $\{X^\mathcal{C}_t\}_{t \in \Zd}$ and $\{X^\mathcal{D}_t\}_{t \in \Zd}$ are two independent stationary S$\alpha$S random fields, $\{X^\mathcal{D}_t\}_{t \in \Zd}$ is a mixed moving average, and $\{X^\mathcal{C}_t\}_{t \in \Zd}$ has no nontrivial mixed moving average component (since it is generated by a conservative $\Zd$-action).

\begin{thm} The decomposition \eqref{decomp_X_t_rosinski} is unique is law, i.e., the (finite-dimensional) distributions of $\{X^\mathcal{C}_t\}_{t \in \Zd}$ and $\{X^\mathcal{D}_t\}_{t \in \Zd}$ do not depend on the choice of Rosi\'{n}ski representation.
\end{thm}

\begin{proof} See the proof of Theorem~4.3 in \cite{rosinski:1995}.
\end{proof}

Thanks to the above result, we define $\{X^\mathcal{C}_t\}_{t \in \Zd}$ and $\{X^\mathcal{D}_t\}_{t \in \Zd}$ to be the conservative and dissipative parts of $\XZd$, respectively.

\section{The Maxima Sequence} \label{sec:maxima}

In view of the discussions in the beginning of Section~\ref{sec_rosinski_repn} and Remark~\ref{remark_cons_diss_memory} above, we can expect that the maxima sequence $M_n$ grows slowly when the underlying $\Zd$-action is conservative. This is confirmed by the following result.

\begin{thm} \label{thm_M_n_cons_diss} Let $\XZd$ be a stationary S$\alpha$S random field generated by a nonsingular $\Zd$-action $\phit$ on $(S, \mathcal{S},\mu)$ with the corresponding Rosi\'{n}ski representation $\{f_t\}_{t \in \Zd}$ of the form \eqref{repn_rosinski}. Then the following results hold.
\begin{enumerate}
\item If $\phit$ is conservative, then $M_n / n^{d/\alpha} \probconv 0$, and
\item if $\phit$ is not conservative, then $M_n / n^{d/\alpha} \weakconv a_\bX Z_\alpha$,
\end{enumerate}
where $a_\bX >0$ is a constant determined by the dissipative part of $\XZd$ and $Z_\alpha$ is a Fr\'{e}chet type extreme value random variable with distribution function \eqref{form_of_cdf_of_Zalpha}.
\end{thm}

The main tool behind the proof of the above result is the deterministic sequence
\begin{equation}
b_n = \left(\int_S \max_{t \in B_n}|f_t(s)|^\alpha \mu(ds)\right)^{1/\alpha}, \label{defn_bn}
\end{equation}
where $B_n$ is as defined in Section~\ref{sec_rosinski_repn}. The first step of the proof is the computation of asymptotics of $b_n$ as $n \to \infty$ and the second step is to show that the asymptotic behaviour of the maxima sequence $M_n$ is more or less determined by that of $b_n$.

\begin{remark} \label{remark_bn_ind_of_repn} \textnormal{By Corollary 4.4.6 of \cite{samorodnitsky:taqqu:1994},}
\begin{equation}
\lim_{\lambda \to \infty} \lambda^\alpha P(M_n > \lambda) = C_\alpha b_n^\alpha, \nonumber
\end{equation}
\textnormal{where $C_\alpha$ is the stable tail constant \eqref{defn_C_alpha}. In particular, this means that the sequence $b_n$ is solely determined by the S$\alpha$S random field $\XZd$ and does not depend on the choice of integral representation $\{f_t\}_{t\in\Zd}$.}
\end{remark}

The first step of the proof of Theorem~\ref{thm_M_n_cons_diss} is given by the following lemma.

\begin{lemma} \label{lemma_asymp_of_bn} Let $\phit$ be as in Theorem~\ref{thm_M_n_cons_diss} and $b_n$ be as in \eqref{defn_bn}. Then the following asymptotic results hold.
\begin{enumerate}
\item If $\phit$ is conservative, then $b_n / n^{d/\alpha} \to 0$, and
\item if $\phit$ is not conservative, then $b_n / n^{d/\alpha} \to K_\bX$,
\end{enumerate}
where $K_\bX >0$ is a constant determined by the dissipative part of $\XZd$.
\end{lemma}

\begin{proof} For the first part, see the proof of Proposition 4.1 in \cite{roy:samorodnitsky:2008}. For the second part, see the proof in the one-dimensional case, i.e., Theorem 3.1 of \cite{samorodnitsky:2004a} (the same proof goes through in the higher dimensional case due to Theorem~\ref{thm_diss_mixed_mov_avg} above).
\end{proof}

The second step of the proof of Theorem~\ref{thm_M_n_cons_diss} relies on the following lemma, which can be established by applying Theorem~\ref{thm_series_repn_SalphaS} on each linear combination of the random vectors.

\begin{lemma} \label{lemma_series_repn} Fix a positive integer $n$. The random vector $(X_t, t \in B_n)$ has a series representation (in law) of the form
\[
(X_t)_{t \in B_n} \eqlaw \left(b_n C_\alpha^{1/\alpha} \sum_{j=1}^\infty \epsilon_j \Gamma_j^{-1/\alpha} \frac{f_t(U^{(n)}_j)}{\max_{ v \in B_n} f_v(U^{(n)}_j)}\right)_{t \in B_n},
\]
where $b_n$ is as in \eqref{defn_bn}, $C_\alpha$ is as in \eqref{defn_C_alpha}, $\{\epsilon_i\}_{i \geq 1}$ and $\{\Gamma_i\}_{i \geq 1}$ are as in Theorem~\ref{thm_series_repn_SalphaS} above, and $\{U^{(n)}_j\}_{j \geq 1}$ is a sequence of i.i.d. $S$-valued random variables with common law
\[
P\big(U^{(n)}_1 \in A\big)= b_n^{-\alpha} \int_A \max_{t \in B_n} |f_t(s)|^\alpha \mu(ds), \;A \in \mathcal{S}.
\]
\end{lemma}

\begin{proof}[Sketch of Proof of Theorem~\ref{thm_M_n_cons_diss}] When $\phit$ is conservative, using Lemma~\ref{lemma_asymp_of_bn} and Lemma~\ref{lemma_series_repn} and a nice coupling argument, it is possible to show that $M_n/n^{d/\alpha} \probconv 0$. See pg 1450 - 1452 of \cite{samorodnitsky:2004a} for the details.

On the other hand, when $\phit$ is not conservative, using Lemma~\ref{lemma_series_repn} above, we have that for any $\lambda>0$,
\begin{align*}
P\left(\frac{M_n}{b_n} > \lambda\right) &= P\left(\max_{t \in B_n} \left|C_\alpha^{1/\alpha} \sum_{j=1}^\infty \epsilon_j \Gamma_j^{-1/\alpha} \frac{f_t(U^{(n)}_j)}{\max_{ v \in B_n} f_v(U^{(n)}_j)}\right|\, > \, \lambda\right),\\
\intertext{from which by using ``one large jump'' principle (see Remark~\ref{remark_series_repn_SalphaS} above), we get}
                                                                    &\approx P\left(\max_{t \in B_n} \left|C_\alpha^{1/\alpha}\epsilon_1 \Gamma_1^{-1/\alpha} \frac{f_t(U^{(n)}_1)}{\max_{ v \in B_n} f_v(U^{(n)}_1)}\right|\, > \, \lambda\right)\\
                                                                    &=P(C_\alpha^{1/\alpha} \Gamma_1^{-1/\alpha} > \lambda) = 1 - e^{-C_\alpha \lambda^{-\alpha}}.
\end{align*}
The above heuristic calculations show that $M_n/b_n \weakconv C_\alpha^{1/\alpha} Z_\alpha$ and the second part of Theorem~\ref{thm_M_n_cons_diss} follows using Lemma~\ref{lemma_asymp_of_bn}. See pg 1454 - 1455 of \cite{samorodnitsky:2004a} to find out how to make the above ``$\approx$'' precise when $\phit$ is not conservative.
\end{proof}

\normalsize

\section{Connections to Finitely Generated Abelian Groups} \label{sec:fg_ab_gp}

As long as the underlying nonsingular action is not conservative, the exact asymptotic behaviour of $M_n$ is given in Theorem~\ref{thm_M_n_cons_diss}. Therefore, more interesting examples of S$\alpha$S random fields are the ones generated by conservative actions. We look at a few of those in this section.

\begin{example} \textnormal{Consider the conservative action in Remark~\ref{exc_example_cons_action}. Choose $\mu$ to be a probability measure on $\mathbb{R}^{\Zd}$ such that under $\mu$, the coordinate field $\{\pi_t\}_{t \in \Zd}$ forms a collection of i.i.d. random variables. In this case, define an S$\alpha$S random field $\XZd$ by}
\[
\{X_t\}_{t \in \Zd} \eqlaw \left\{\int_{\mathbb{R}^{\Zd}} \pi_0 \circ \zeta_t (\bx) dM(\bx)\right\}_{t \in \Zd},
\]
\textnormal{where $M$ is an S$\alpha$S random measure on $\mathbb{R}^{\Zd}$ with control measure $\mu$ and other notations are as in Remark~\ref{exc_example_cons_action}.}

\textnormal{If further, we assume that $\pi_0$ follows standard normal distribution under $\mu$, then it would follow that $\XZd$ is a sub-Gaussian random field, i.e., there is a collection of i.i.d. standard normal random variables $\{\xi_t\}_{t \in \Zd}$ and another independent positive stable random variable $A$ defined on the same probability space such that}
\[
\{X_t\}_{t \in \Zd} \eqlaw \{A\xi_t\}_{t \in \Zd}.
\]
\textnormal{See Proposition 3.7.1 in \cite{samorodnitsky:taqqu:1994}. Using this sub-Gaussian representation and standard extreme value theory estimates (see, for example, \cite{resnick:1987}), it follows that}
\[
\frac{M_n}{\sqrt{2d\log{n}}} \weakconv A,
\]
\textnormal{a non-extreme value limit.}

\textnormal{On the other hand, if $\pi_0$ follows Pareto distribution with parameter $\theta > \alpha$ (i.e., $\mu(\pi_0>x)=x^{-\theta},\,x \geq 1$), then it can be shown that}
\[
\frac{M_n}{n^{d/\theta}} \weakconv c_{\alpha, \theta} Z_\alpha,
\]
\textnormal{for some finite positive constant $c_{\alpha, \theta}$; see Section 5 in \cite{samorodnitsky:2004a} for the details.}
\end{example}

The above example shows that in the conservative case, the rate of growth of the partial maxima sequence can be either polynomial or slowly varying. Heuristically, one can say that stronger conservativity of the underlying group action should imply longer memory, which in turn should give rise to slower rate of growth of $M_n$. Therefore, the following question becomes pertinent in the setup of Rosi\'{n}ski representations of stationary S$\alpha$S random fields.

\begin{qn} How to quantify the ``strength of conservativity'' of the underlying nonsingular $\Zd$-action?
\end{qn}

In general the answer to the above question is not known. However, \cite{roy:samorodnitsky:2008} made further investigations on
the actual rate of growth of the partial maxima sequence $M_n$ using
the theory of finitely generated abelian groups (see, for example,
\cite{lang:2002}) together with counting of the number of lattice
points in dilates of rational polytopes (see \cite{deloera:2005}). Viewing the action as a
group of nonsingular transformations and studying the algebraic
structure of this group, one can get better ideas about the strength of conservativity of the underlying action and hence the
rate of growth of the partial maxima as well as the length of memory
of the random field. We start with the following motivating example.

\begin{example} \label{example_motivating_ab_gp} \textnormal{Let $S=\mathbb{R}$, $\mu = Leb$, $d=2$, and $\{\phi_{(i,j)}\}_{(i,j) \in \mathbb{Z}^2}$ be the measure-preserving $\mathbb{Z}^2$-action on $\mathbb{R}$ defined by
$\phi_{(i,j)}(s)=s+i-j$, $s \in \mathbb{R}$. Take any $f \in L^\alpha(\mathbb{R}, Leb)$ and define a stationary S$\alpha$S random field by}
\[
\{X_{(i,j)}\}_{(i,j) \in \mathbb{Z}^2} = \left\{\int_\mathbb{R} f (\phi_{(i,j)}(s)) M(ds) \right\}_{(i,j) \in \mathbb{Z}^2},
\]
\textnormal{where $M$ is an S$\alpha$S random measure on $\mathbb{R}$ with control measure $\mu = Leb$. Fix $k \in \mathbb{Z}$. Note that for each $(i, j) \in \mathbb{Z}^2$ situated on the line $j=i+k$, $\phi_{(i,j)}=\phi_{(0,k)}$ and therefore $X_{(i,j)} =X_{(0,k)}$ almost surely.}
\textnormal{Therefore using stationarity of $\XZd$, we have}
\[
M_n = \max_{0 \leq i,j\leq n-1} |X_{(i,j)}| \eqas \max_{1-n \leq k\leq n-1} |X_{(0,k)}| \eqlaw \max_{0 \leq k\leq 2(n-1)} |X_{(0,k)}|
\]
\textnormal{for all $n \geq 1$. Since $\{X_{(0,k)}\}_{k \in \mathbb{Z}}$ is a stationary S$\alpha$S process generated by the dissipative $\mathbb{Z}$-action $\{\phi_{(0,k)}\}_{k \in \mathbb{Z}}$, we get that there exists a constant $a>0$ such that}
\[
\frac{M_n}{n^{1/\alpha}} \weakconv a Z_\alpha.
\]
\end{example}

In the example above, we see a reduction of ``effective dimension'' of the random field. Algebraically, this boils down to quotienting $\mathbb{Z}^2$ by the diagonal $K=\{(i,j) \in \mathbb{Z}^2: i=j\}$. Note that $K$ is the kernel of the group homomorphism $(i,j) \mapsto \phi_{(i,j)}$. Reduction of dimension occurs because $\mathbb{Z}^2/K \simeq \mathbb{Z}$.

In general, if a stationary S$\alpha$S random field $\XZd$ is generated by a nonsingular $\Zd$-action $\phit$, then we need to look at the kernel $K$ of the group homomorphism $t \mapsto \phi_t$, i.e.,
$$
K:=\{t \in \Zd: \phi_t (s) = s \mbox{ for $\mu$-almost all }s \in S\}.
$$
In general, it may not happen that $\Zd/K \simeq \mathbb{Z}^p$ for some $p \leq d$. However by Structure Theorem for Finitely Generated Abelian Groups (see Theorem 8.5 in Chap.~I of \cite{lang:2002}),
\[
\Zd/K = \bar{F} \oplus \bar{N},
\]
where $\bar{F} \simeq \mathbb{Z}^p$ for some $p \leq d$ and $\bar{N}$ is a finite group. Here $\oplus$ denotes the direct sum of groups. Using the fact that $\bar{F}$ is a free abelian group, it is possible to show that $\bar{F}$ has an isomorphic copy $F$ sitting inside $\Zd$; see Section~5 of \cite{roy:samorodnitsky:2008}. Fix such an $F$. In this setup, $p$ plays the role of ``effective dimension'' and $F$ plays the role of ``effective index set'' of the random field.

In Example~\ref{example_motivating_ab_gp}, $d=2$, $K=\{(i,j) \in \mathbb{Z}^2: i=j\}$, $p=1$ and $\bar{N}$ is trivial. In this case, the ``effective index set'' can be chosen to be $F=\{(0,k): k \in \mathbb{Z}\}$ and since the restricted action $\{\phi_{(i,j)}\}_{(i,j) \in F} = \{\phi_{(0,k)}\}_{k \in \mathbb{Z}}$ is dissipative, we get $M_n/n^{1/\alpha} \weakconv a Z_\alpha$. The general result is as follows.

\begin{thm} \label{thm_M_n_ab_gp} In the above setup, assume that $1 \leq p <d$. Then the following results hold.
\begin{enumerate}
\item If $\{\phi_t\}_{t \in F}$ is conservative, then $M_n / n^{p/\alpha} \probconv 0$, and
\item if $\{\phi_t\}_{t \in F}$ is not conservative, then $M_n / n^{p/\alpha} \weakconv c_\bX Z_\alpha$,
\end{enumerate}
where $c_\bX >0$ is a constant determined by $\XZd$ and $Z_\alpha$ is a Fr\'{e}chet type extreme value random variable with distribution function \eqref{form_of_cdf_of_Zalpha}.
\end{thm}

\begin{proof} This proof is mostly algebraic with a slight touch of combinatorics in it; see Section 5 of \cite{roy:samorodnitsky:2008}.
\end{proof}

\subsection{Extensions and Open Problems} The discrete parameter results mentioned in this paper have been extended to the continuous parameter stationary measurable locally bounded S$\alpha$S random fields $\{X_t\}_{t \in \mathbb{R}^d}$ in \cite{rosinski:1995, rosinski:2000}, \cite{samorodnitsky:2004b} and \cite{roy:2010b}. The approach taken by them is to approximate the continuous parameter random field $\{X_t\}_{t \in \mathbb{R}^d}$ by its discrete parameter skeletons $\{X_t\}_{t \in  2^{-i}\mathbb{Z}^d}$, $i=0,1,2, \ldots$. In \cite{chakrabarty:roy:2013}, the notion of effective dimension was extended to the continuous parameter case based on the following observation: the effective dimensions of $\{X_t\}_{t \in  2^{-i}\mathbb{Z}^d}$, $i=0,1,2, \ldots$ are all equal and therefore can be defined as the effective dimension of $\{X_t\}_{t \in \mathbb{R}^d}$. With this definition, Theorem~\ref{thm_M_n_ab_gp} was extended to the continuous parameter case.

\cite{owada:samorodnitsky:2015a} used the sophisticated machinery of pointwise dual ergodicity (see, for example, \cite{aaronson:1997}) to derive a functional limit theorem (in Skorohod's $M_1$-topology and in some cases, $J_1$-topology) for the scaled partial maxima process based on a stationary S$\alpha$S process generated by a measure-preserving conservative action. However, generalization of this work (and also \cite{owada:samorodnitsky:2015b}, \cite{jung:owada:samorodnitsky:2015}, \cite{owada:2016}) to random fields is open and requires the notion of pointwise dual ergodicty in the multiparameter case.

Recently, \cite{sarkar:roy:2016} investigated a similar maxima sequence for stationary S$\alpha$S random fields indexed by finitely generated free groups and obtained a different phase transition boundary between shorter and longer memory. In particular, they have produced an example of such a random field induced by a conservative action of the free group but its maxima sequence grows as fast as the i.i.d. field as opposed to what happens in the case of $\Zd$. A deeper connection to algebra (as in Theorem~\ref{thm_M_n_ab_gp} above) is still missing mainly because of unavailability of a general structure theorem for finitely generated noncommutative groups. It is perhaps possible to resolve this issue in special classes of actions but nothing is clear at the moment.

\section{Summary of Related Work} \label{sec:other_works}

A few important classes of stationary S$\alpha$S processes were introduced in \cite{rosinski:samorodnitsky:1996} and \cite{cohen:samorodnitsky:2006}.

Various probabilistic aspects of stationary S$\alpha$S random fields and processes have also been connected to the ergodic theoretic properties of the underlying nonsingular action. \cite{mikosch:samorodnitsky:2000a} investigated the ruin probabilities of a negatively drifted random walk whose steps are coming from a stationary ergodic stable process and observed that ruin becomes more likely when the underlying $\mathbb{Z}$-action is conservative.

The point process induced by stationary S$\alpha$S processes was considered in \cite{resnick:samorodnitsky:2004} and this work was extended to the random fields in \cite{roy:2010a}. It was seen that when the underlying action is not conservative, the associated point process sequence converges weakly to a Poisson cluster process. However in the conservative case, the point process sequence does not remain tight due to clustering. In many such examples, the point process sequence can be shown to converge to a random measure after proper normalization. In particular, the connection to finitely generated abelian groups carries forward to this setup as well.

\cite{fasen:roy:2016} investigated the large deviation behaviour of a point process sequence
induced by a stationary S$\alpha$S random field based on the framework introduced in \cite{hult:samorodnitsky:2010}. Once again, depending on the ergodic theoretic and group theoretic structures of the underlying nonsingular $\Zd$-action, different large deviation behaviours were observed. This was used to study the large deviations of maxima and partial sum sequences of such fields.

Using the language of positive-null decomposition of nonsingular flows (see Section 1.4 in \cite{aaronson:1997} and Section 3.4 in \cite{krengel:1985}) another decomposition of measurable stationary $S\alpha S$ processes was obtained in \cite{samorodnitsky:2005a} and this decomposition was used to characterize the ergodicity of such a process. This work was extended to the stationary S$\alpha$S random fields in \cite{wang:roy:stoev:2013} based on the work \cite{takahashi:1971}. See also \cite{roy:2012} for another recent work connecting Maharam systems with various ergodic properties of stationary stable processes.

A systematic and wholesome approach to decompositions of a stationary S$\alpha$S process into independent stationary S$\alpha$S components is presented in \cite{wang:stoev:roy:2012}.

Decompositions based on the ergodic theory of nonsingular actions were also obtained for self-similar $S\alpha S$ processes with stationary
increments in \cite{pipiras:taqqu:2002a} and \cite{pipiras:taqqu:2002b}. See also \cite{kolodynski:rosinski:2003} for existence and rigidity results for integral representations of group self-similar stable processes.

Many of the results mentioned in this survey have parallels in the max-stable world. See, for example, \cite{stoev:taqqu:2005}, \cite{stoev:2008}, \cite{kabluchko:2009}, \cite{kabluchko:schlather:dehaan:2009}, \cite{kabluchko:schlather:2010}, \cite{wang:stoev:2010a}, \cite{wang:stoev:2010b}, \cite{wang:stoev:roy:2012}, \cite{wang:roy:stoev:2013}, \cite{dombry:kabluchko:2014}, \cite{dombry:kabluchko:2016}. For links between stationary infinitely divisible processes and ergodic theory, see \cite{roy:2007a}, \cite{roy:2009}, \cite{owada:samorodnitsky:2015b}, \cite{jung:owada:samorodnitsky:2015}, \cite{owada:2016}, \cite{kabluchko:stoev:2016}.

\section*{Acknowledgements} The author would like to express his sincere gratitude towards Professor B.~V.~Rao, whose outstanding teaching and vivacious personality have always been a source of inspiration. He would also like to thank the organisers of \emph{Eighth Lectures on Probability and Stochastic Processes} for the kind invitation and warm hospitality, and the participants for pointing out typos and mistakes in the first draft of the lecture notes. Last but not the least, the careful reading and detailed comments by the anonymous referee are gratefully acknowledged.

\end{document}